\documentclass[titlepage,a4paper,reqno,11pt]{amsart}

\usepackage{amsaddr}

\usepackage{preamble}

\title{The Dimension of Divisibility Orders and Multiset Posets}
\author{Milan Haiman}\address{Massachusetts Institute of Technology}\email{mhaiman@mit.edu}

\begin{document}

\begin{abstract}
    The Dushnik--Miller dimension of a poset $P$ is the least $d$ for which $P$ can be embedded into a product of $d$ chains. Lewis and Souza showed that the dimension of the divisibility order on the interval of integers $[N/\kappa, N]$ is bounded above by $\kappa (\log\kappa)^{1+o(1)}$ and below by $\Omega((\log\kappa/\log\log\kappa)^2)$. We improve the upper bound to $O((\log \kappa)^3/(\log\log\kappa)^2).$ We deduce this bound from a more general result on posets of multisets ordered by inclusion. We also consider other divisibility orders and give a bound for polynomials ordered by divisibility.
\end{abstract}


\maketitle

\section{Introduction}\label{sec:introduction}

A \emph{partially ordered set} (abbreviated \emph{poset}) is an ordered pair $(P,\le_P)$ consisting of a set $P$ and a binary relation $\le_P$ on $P$ such that for all $a,b,c\in P$ we have \begin{itemize}
    \item $a\le_P a$,
    \item if $a\le_P b$ and $b\le_P a$, then $a=b$, and
    \item if $a\le_P b$ and $b\le_P c$, then $a\le_P c$.
\end{itemize}
We will refer to a poset $(P,\le_P)$ by just $P$ if the order $\le_P$ is clear from context. In such cases we may also write $\le$ instead of $\le_P$. We will also write $a<b$ to denote $a\le b$ and $a\ne b$. Finally, we will only work with finite posets unless explicitly stated otherwise. For a finite poset $P$, we denote the cardinality of its underlying set by $\abs{P}$.

For a poset $P$ with elements $a$ and $b$, we say that $a$ and $b$ are \emph{comparable} if $a\le b$ or $b\le a$. Otherwise, we say that $a$ and $b$ are \emph{incomparable}. The simplest example of a poset is one in which all pairs of elements are comparable. We call such a poset a \emph{chain} or \emph{linear order} (also known as a \emph{total order}). 

A more complicated example of a poset is the divisibility poset $\D_{[6]}$, which consists of the set $[6]=\{1,2,3,4,5,6\}$ with the relation $a\le b$ if $b$ is divisible by $a$. Note that some pairs of elements are incomparable, such as $5$ and $6$. This poset is depicted in \cref{fig:D6-Hasse}, where $a\le b$ if and only if $a$ is reachable from $b$ by a sequence of downward arrows.

\begin{figure}[h]
    \centering
    \includegraphics[scale=0.9,trim=1in 7.65in 1in 1.65in,clip]{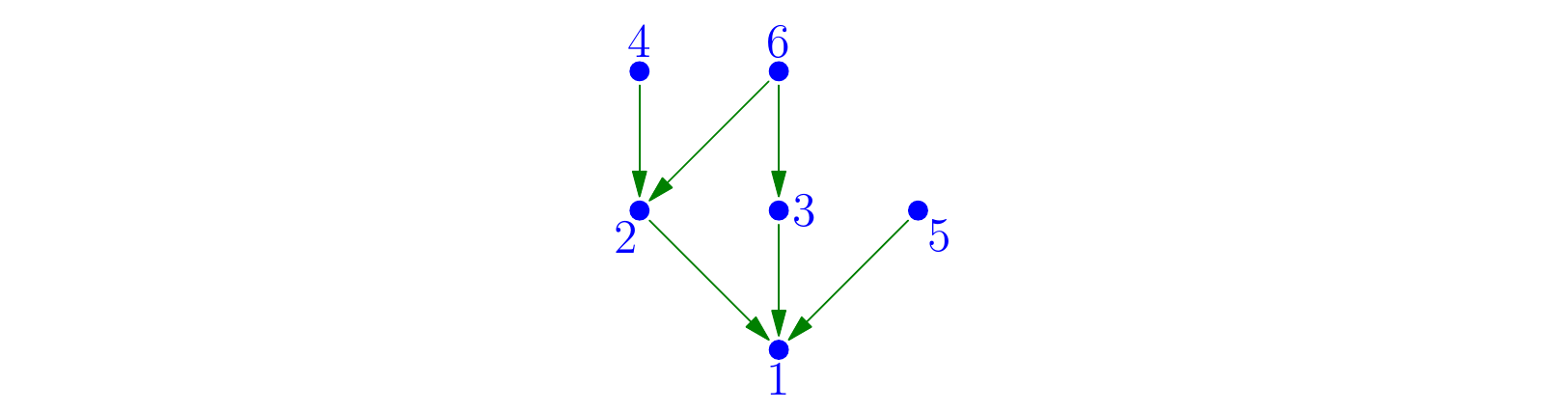}
    \caption{The Hasse diagram of $\D_{[6]}$}
    \label{fig:D6-Hasse}
\end{figure}



In general, for a poset $P$, we call a minimal such diagram the \emph{Hasse diagram} of $P$. The arrows are defined by \emph{covering relations}: $b$ \emph{covers} $a$ if $a < b$ and there is no $c$ such that $a<c<b$.

Given a poset $P$, a natural question to ask is how complicated $P$ is. A notion such as the set cardinality of $P$ would not be a good answer, since it does not consider the relation $\le_P$. For example, a linear order on $[6]$ is intuitively much simpler than $\D_{[6]}$. A potential approach is to try to understand how complicated the Hasse diagram of $P$ is. Specifically, we could ask for the least $d$ such that the Hasse diagram of $P$ can be viewed ``nicely'' in $d$-dimensional space. For instance, the linear order on $[6]$ can be viewed as a line, while $D_{[6]}$ can't. We formalize this idea of dimension using some more definitions.

Given posets $P$ and $Q$, we say that $P$ \emph{embeds} into $Q$, written $P\into Q$, if $P$ is a subset of $Q$ and $\le_P$ is the restriction of $\le_Q$ to $P$. We also say that $P$ is a \emph{suborder} of $Q$.

This notion of poset containment allows us to relate a poset to $d$-dimensional space. We just need to think of $d$-dimensional space as a poset. For this, we define the product of posets.

Given posets $P$ and $Q$, we define the \emph{product poset} $P\times Q$ to be a poset on the set product $P\times Q$ with the relation $(p_1,q_1)\le (p_2,q_2)$ if and only if $p_1\le_Pp_2$ and $q_1\le_Qq_2$. We can also take the product of several posets by iterating this definition.

Now, we can think of $d$-dimensional space as the poset $\R^d$, where we take the usual linear order on $d$ copies of $\R$ and then take a product. More generally, we can think of $d$-dimensional space as a product of any $d$ chains. Since we will only care about finite posets, we can work with finite chains as well. We are now ready to define the dimension of a poset.

The \emph{Dushnik--Miller dimension} of a poset $P$, denoted $\dim(P)$, is the least $d\in\N$ for which $P$ embeds into a product of $d$ chains. Note that dimension exists for all finite posets since we can always embed $P$ into a product of $\abs{P}$ chains. This notion of dimension was introduced by Dushnik and Miller in 1941 \cite{Dushnik-Miller-paper} and has been extensively studied (see \cite{Trotter-book}). Other notions of dimension exist (see \cite{Nesetril-Pudlak-paper,Novak-paper}), but we will only discuss the Dushnik--Miller dimension of posets.

As an example, any linear order has dimension $1$. A more interesting example is the poset $\D_{[6]}$ from above. One can check that $\dim(\D_{[6]})>1$. Furthermore, $\dim(\D_{[6]})\le2$ by the embedding in \cref{fig:D6-dim} of $\D_{[6]}$ into a $4$ by $3$ grid (which is a product of chains of length $4$ and $3$).

\begin{figure}[h]
    \centering
    \includegraphics[scale=0.9,trim=1in 7.35in 1in 1.7in,clip]{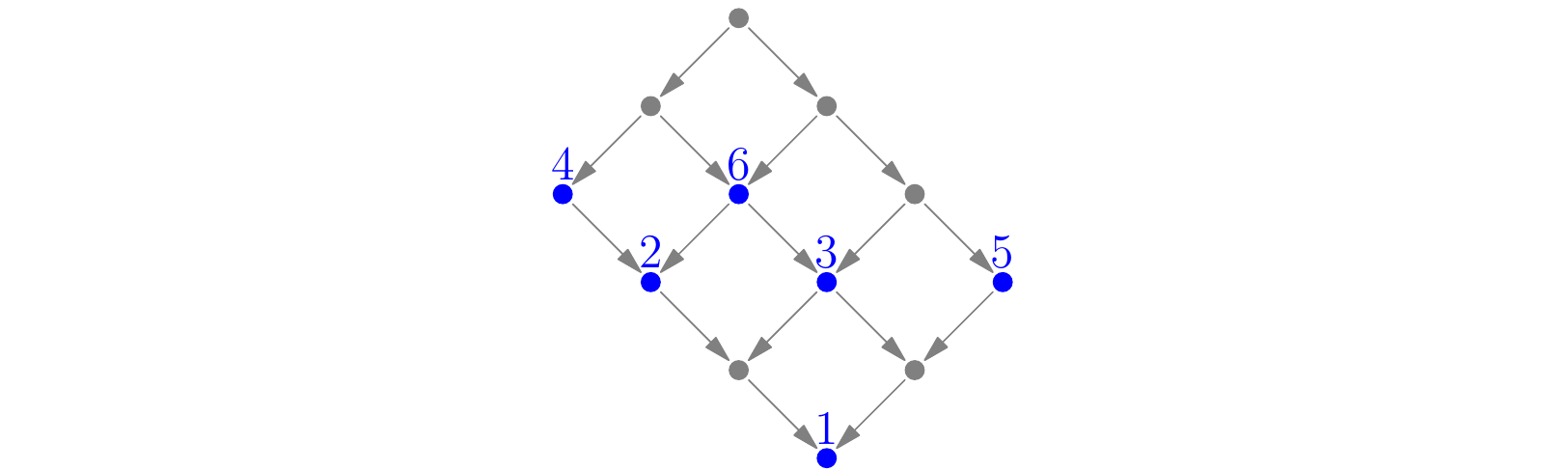}
    \caption{An embedding of $\D_{[6]}$ into a product of two chains}
    \label{fig:D6-dim}
\end{figure}

A more general example is the Boolean hypercube $\Q^n$. We view $\Q^n$ as the poset on subsets of $[n]$, ordered by inclusion (i.e., the ``subset'' relation). Since $\Q^n$ is generally thought of as an $n$-dimensional object, we would expect it to have dimension $n$. In fact, because $\Q^n$ is isomorphic to a product of $n$ chains of length $2$, we have that $\dim(\Q^n)\le n$. We will see in \cref{sec:prelimaries} that indeed $\dim(\Q^n)=n$. Things get more interesting when we consider suborders of $\Q^n$.

Note that $\Q^n$ has a natural partition into ``layers,'' where we partition subsets of $[n]$ by their cardinality. The Hasse diagram of $\Q^n$ can be drawn with each layer being on a horizontal line. So, a natural way to take a suborder of $\Q^n$ is to just consider some subset of its layers. For $I\subseteq [0,n]$, we let $\Q^n_I$ be the suborder of $\Q^n$ on subsets $S$ of $[n]$ for which $\abs{S}\in I$. The dimension of such suborders has been studied in detail; see \cite{Kierstead-overview} for a survey by Kierstead.

We will see in \cref{sec:prelimaries} that for $0<k<\ell<n$, $\dim(\Q^n_{\{k,\ell\}})=\dim(\Q^n_{[k,\ell]})$. That is, if we take a suborder on some layers of the hypercube, it only matters what the lowest and highest layers are. In particular, to understand $\dim(\Q^n_I)$ for arbitrary $I$ we just need to understand $\dim(\Q^n_{\{k,\ell\}})$. Several results are known, and we present some of them here. The case $k=1$ was studied first. Dushnik gave an exact formula for $\dim(\Q^n_{[1,\ell]})$ when $\ell>2\sqrt{n}$ in \cite{Dushnik-paper}. In particular, we have the following theorem.

\begin{theorem}[Dushnik]
If $\ell>2\sqrt{n}$, then $\dim(\Q^n_{[1,\ell]})>n-\sqrt{n}.$
\end{theorem}

By using similar techniques, one can obtain the following corollary \cite{Trotter-book}.

\begin{corollary}[Trotter] 
If $\ell<\sqrt{n},$ then $\dim(\Q^n_{[1,\ell]})>\frac{\ell^2}{4}.$
\end{corollary}

Together, these results show that 
$\dim(\Q^n_{[1,\ell]})$ grows at least quadratically in $\ell$ until it gets close to the upper bound of $\dim(\Q^n_{[1,n]})=n$.

A result by F\"{u}redi and Kahn shows that this quadratic growth is correct up to a $\log n$ factor~\cite{Furedi-Khan-paper}. Note that we use $\log$ to denote the natural logarithm throughout this paper.

\begin{theorem}[F\"{u}redi--Kahn]\label{thm:weak-subsets}
For integers $1\le \ell\le n$, we have $\dim(\Q^n_{[1,\ell]})\le(\ell+1)^2\log n.$
\end{theorem}

For the case of general $k$, the lower bounds for $k=1$ also imply that $\dim(\Q^n_{[k,\ell]})$ grows at least quadratically in $(\ell-k)$, by noting that $\Q^n_{[k,\ell]}$ contains a copy of $\Q^{n-k+1}_{[1,\ell-k+1]}$

In 1994, Brightwell, Kierstead, Kostochka, and Trotter gave the following upper bound \cite{BKKT-paper}. 

\begin{theorem}[Brightwell--Kierstead--Kostochka--Trotter]\label{thm:subsets}
Let $0\le k< \ell\le n$ be integers. Then we have $\dim(\Q^n_{[k,\ell]})=O((\ell-k)^2\log n).$
\end{theorem}
This shows that $\dim(\Q^n_{[k,\ell]})$ grows quadratically in $(\ell-k)$, up to a $\log n$ factor. These bounds are relatively tight when $\ell-k$ is large compared to $\log n$. For results on the case when $\ell-k$ is smaller, see \cite{Kierstead-paper,Spencer-paper}.

A natural extension of this problem on subsets is to instead consider multisets of $[n]$. Let $\M^n$ be the (infinite) poset on all multisets of $[n]$, ordered by inclusion. As before, for $I\subseteq \N$, let $\M^n_I$ be the suborder of $\M^n$ on multisets $S$ of $[n]$ for which $\abs{S}\in I$ (in this paper $I$ will always be finite). Here $\abs{S}$ counts elements of $S$ with multiplicity. 

Our first new result is the following extension of \cref{thm:subsets}.\footnote{We have not optimized the constants in new results in this paper. We also omit some floors and ceilings.}

\begin{theorem}\label{thm:multisets} Let $k,\ell,n$ be positive integers with $k< \ell$. Then we have
$$\dim(\M_{[k,\ell]}^n)\le 34(\ell-k)^2\log n.$$
\end{theorem}

Next, we consider a problem posed by Lewis and Souza in 2021. For a set $R$ of positive integers, let $\D_R$ be the poset on $R$, ordered by divisibility. Earlier we saw an example with $R=[6]$ and showed that $\dim(\D_{[6]})=2$. More generally, one can ask how $\dim(\D_R)$ behaves in terms of $R$. A natural choice for $R$ is the interval $[N]$. Lewis and Souza (with an improvement to the upper bound due to Souza and Versteegen) essentially solved this case \cite{Lewis-Souza-paper, improved-divisibility}.

\begin{theorem}[Lewis--Souza, Souza--Versteegen] Let $N$ be a positive integer. Then we have
$$\left(\frac{1}{16}-o(1)\right)\frac{(\log N)^2}{(\log\log N)^2}\le \dim(\D_{[N]})\le \left(\frac{4}{\log2}+o(1)\right)\frac{(\log N)^2\log\log\log N}{(\log\log N)^2}.$$
\end{theorem}

Another interesting choice of $R$ is $R=[N/\kappa,N]$, where $\kappa> 1$ is a real number and $N$ is a positive integer. Lewis and Souza noted that $\D_{[1,\kappa]}$ embeds into $\D_{[N/\kappa,N]}$, which gives a lower bound of $$\dim(\D_{[N/\kappa,N]})=\Omega((\log \kappa)^2(\log\log \kappa)^{-2}).$$

Lewis and Souza also showed that $\dim(\D_{[N/\kappa,N]})$ is bounded above by $\kappa (\log\kappa)^{1+o(1)}$, using a general result of Scott and Wood \cite{Scott-Wood-paper}. Note that this is a function of only $\kappa$ and not $N$. So, a natural question is to understand the behavior of $\sup_N\dim(\D_{[N/\kappa,N]})$ as a function of $\kappa$.
One of our main results in this paper is to significantly improve the above upper bound with the following theorem.

\begin{theorem}\label{thm:divisibility} Let $\kappa>1$ be a real number and let $N$ be a positive integer. Then we have
$$\dim(\D_{[N/\kappa,N]})\le \max\left(688\frac{(\log\kappa)^3}{(\log\log \kappa)^2},2\right).$$
\end{theorem}

The $2$ in the $\max$ function handles the case when $\kappa$ is less than $3$. To prove this theorem, we first generalize \cref{thm:multisets}. Note that $\M^n$ is isomorphic to $\Z_{\ge0}^n$, by identifying a multiset $S$ with its tuple of multiplicities $\x(S)$. Additionally, we can identify an integer with the tuple of exponents in its prime factorization. However, when we consider the size of an integer, each prime has a different weight. Thus, to use multisets to capture an interval of integers, we need to weight the elements of $[n]$ in our multisets. So, for a weight vector $\vv\in\R_{>0}^n$, we define the \emph{$\vv$-size} of a multiset $S$ to be $\abs{S}_{\vv}=\x(S)\cdot \vv$. We also define 
$$\M_{[k,\ell]}^{n,\vv}=\{S\in\M^n \st \abs{S}_{\vv} \in [k,\ell]\}.$$
Here we allow $k$ and $\ell$ to be real numbers, as $\abs{S}_{\vv}$ may not be an integer. Note that when all entries of $\vv$ are $1$ and $k,\ell$ are integers, $\M_{[k,\ell]}^{n,\vv}=\M_{[k,\ell]}^{n}$.

To state our result on weighted multisets, we first need a technical definition to describe the dependence on $\vv$. For a vector $\vv\in\R_{>0}^n$ and a real $s$, we let $m(\vv,s)$ be the sum of the least $\floor{s}$ coordinates of $\vv$. Then we have the following result for weighted multisets, which generalizes \cref{thm:multisets}.

\begin{theorem}\label{thm:weighted-multisets}
Let $0\le k<\ell$ be real numbers and let $n$ be a positive integer. Let $\vv\in \R_{>0}^n$ be a weight vector. Let $r$ be the least real number such that $m(\vv,r)\ge 2(\ell-k)$ and $r\ge1$. Then we have
$$\dim(\M_{[k,\ell]}^{n,\vv})\le 43r^2\log n.$$

\end{theorem}

The key new idea in this paper is a way to deal with arbitrary weight vectors $\vv$ in the proof of \cref{thm:weighted-multisets}. We will prove \cref{thm:divisibility} by applying \cref{thm:weighted-multisets}.

\cref{thm:weighted-multisets} can also be used to analyze other divisibility orders. For example, consider divisibility orders on polynomials. For a prime power $q$ and integers $d_0\ge0$, $\delta>0$, let $P(\F_q)_{[d_0-\delta,d_0]}$ denote the poset on monic polynomials in $\F_q[x]$ with degree in the interval $[d_0-\delta,d_0]$, ordered by divisibility. Using \cref{thm:weighted-multisets}, we obtain the following result.

\begin{theorem}\label{thm:polynomial} Let $q$ be a prime power and let $\delta$ be a positive integer. Then for each nonnegative integer $d_0$, we have

$$\dim(P(\F_q)_{[d_0-\delta,d_0]})\le \min\left(910\frac{(\delta\log q)^3}{(\log \delta)^2},172\delta^3\log q\right).$$
\end{theorem}

The two bounds in this theorem correspond to the two regimes when $q<\delta$ and $q\ge \delta$.

The paper is organized as follows. In \cref{sec:prelimaries} we will discuss some basic properties of poset dimension that will be useful in the proofs of the main results. In \cref{sec:multisets} we analyze unweighted multiset posets and prove \cref{thm:multisets}. In \cref{sec:weighted-multisets} we extend our arguemnts to weighted multiset posets and prove \cref{thm:weighted-multisets}. In \cref{sec:divisibility} we prove \cref{thm:divisibility} by applying \cref{thm:weighted-multisets}. In \cref{sec:polynomial}, we discuss other divisibility orders and prove \cref{thm:polynomial}. In \cref{sec:further-directions} we discuss further directions.

\section{Preliminaries}\label{sec:prelimaries}

A \emph{linear extension} of a poset $P$ is a total (linear) order on the elements of $P$ that agrees with all relations in $P$. Given a set $\LL$ of linear extensions of $P$, we say that $\LL$ is a \emph{realiser} if for each pair\footnote{We will refer to ordered pairs as pairs when the order is clear.} of incomparable elements ($x,y$), there exists an extension $L\in\LL$ such that $x\ge_L y$. Then $\dim(P)$ is also the minimum size of a realiser of $P$. One way to see this is by embedding $P$ into $\R^d$ and constructing a linear extension from each of the $d$ coordinates.

We now discuss some basic properties of dimension. Given two posets $P$ and $Q$, if $P$ embeds into $Q$, then $\dim(P)\le \dim(Q)$. 

An important concept in the study of poset dimensions is the notion of critical pairs. A \emph{critical pair} in a poset $P$ is a pair of incomparable elements $(x,y)$ such that $x$ is minimal among elements incomparable to $y$ and $y$ is maximal among elements incomparable to $x$. We say that a linear extension $L$ of $P$ \emph{reverses} a critical pair $(x,y)$ in $P$ if $x\ge_Ly$.

\begin{proposition}
Let $\LL$ be a set of linear extensions of a (finite) poset $P$. Suppose that each critical pair in $P$ is reversed by some extension in $\LL$. Then $\LL$ is a realiser of $P$.
\end{proposition}

\begin{proof}
Let $(x,y)$ be any pair of incomparable elements in $P$. As long as $(x,y)$ is not a critical pair, we can decrease $x$ or increase $y$ while keeping the pair incomparable. Since $P$ is finite, this means we can find a critical pair $(x',y')$ such that $x\ge_P x'$ and $y\le_P y'$. Now there exists some $L\in\LL$ such that $x'\ge_Ly'$, so we also have $x\ge_Ly$ as desired.
\end{proof}

Using the concept of a critical pair, we can see why we only care about the lowest and highest layers of a suborder of $\Q^n$. Let $0<k<\ell<n$ and consider the poset $\Q^n_{[k,\ell]}$. In this poset the critical pairs are precisely the incomparable pairs of subsets $(X,Y)$ where $\abs{X}=k$ and $\abs{Y}=\ell$. So to construct a realiser for $\Q^n_{[k,\ell]}$ it suffices to construct a realiser for the suborder $\Q^n_{\{k,\ell\}}$. Thus we have that $\dim(\Q^n_{[k,\ell]})\le\dim(\Q^n_{\{k,\ell\}})$. Additionally, since $\Q^n_{\{k,\ell\}}$ is a suborder of $\Q^n_{[k,\ell]}$, we have that $\dim(\Q^n_{\{k,\ell\}})\le\dim(\Q^n_{[k,\ell]})$. Thus $\dim(\Q^n_{\{k,\ell\}})=\dim(\Q^n_{[k,\ell]})$. 

Critical pairs also explain why $\dim(\Q^n)=n$. The critical pairs in $Q^n$ are precisely the pairs $(\{x\},[n]\setminus\{x\})$ for $x\in [n]$. There are $n$ such critical pairs and a linear extension of $\Q^n$ can reverse at most one of them. Thus $\dim(\Q^n)\ge n$. We already know that $\dim(\Q^n)\le n$, so $\dim(\Q^n)=n$.

Next we consider some basic poset constructions. Given posets $P$ and $Q$, recall the product poset $P\times Q$, where $(p_1,q_1)\le (p_2,q_2)$ if and only if $p_1\le_Pp_2$ and $q_1\le_Qq_2$. Then we have that $\dim(P\times Q)\le \dim(P)+\dim(Q)$ by combining embeddings of $P$ and $Q$ into an embedding of $P\times Q$.

Another poset construction we will use is the disjoint union of two posets. Given posets $P$ and $Q$, let $P\sqcup Q$ be their disjoint union. Here we take the union of a copy of $P$ and a copy of $Q$ with all elements of $P$ being incomparable with all elements of $Q$. As long as $P$ and $Q$ are nonempty, $P\sqcup Q$ cannot be a chain, so $\dim(P\sqcup Q)\ge2$. Next, let $d=\max(\dim(P),\dim(Q),2)$ and note that both $P$ and $Q$ can be embedded into $\R^d$. By appropriately translating these embeddings, we can construct an embedding of $P\sqcup Q$ into $\R^d$. So $\dim(P\sqcup Q)=\max(\dim(P),\dim(Q),2)$.

Next, to illustrate some of the techniques used in the proof of \cref{thm:divisibility}, we will prove the following easier result, which is already a small improvement upon the previous best known bound.

\begin{proposition}\label{prop:easy-divisibility} Let $\kappa>1$ and $N\in \N$. Let $\pi(\kappa)$ denote the number of primes less than or equal to $\kappa$. Then we have
$$\dim(\D_{[N/\kappa,N]})\le \max(\pi(\kappa),2)=(1+o(1))\kappa(\log \kappa)^{-1}.$$
\end{proposition}
\begin{proof}
The key idea is to write $\D_{[N/\kappa,N]}$ as a disjoint union of several posets.

Call a prime $p$ \emph{small} if $p\le \kappa$ and \emph{large} if $p>\kappa$. Let $K$ be the set of all positive integers $M\le N$ with no small primes dividing $M$. For each $M\in K$, let $g(M)$ be the set of all positive integers of the form $Mq$, where $q$ has only small prime divisors. Now, note that we have a set partition of $[N/\kappa,N]$ given by $$[N/\kappa,N]=\bigsqcup_{M\in K} g(M)\cap [N/\kappa,N].$$
Furthermore, elements of different sets in this partition are incomparable in $\D_{[N/\kappa,N]}$, because if they
were comparable, their ratio would be divisible by a large prime. Thus $\D_{[N/\kappa,N]}$ is a disjoint union of posets of the form $\D_{g(M)\cap [N/\kappa,N]}$, and
it suffices to show that $\dim(\D_{g(M)\cap [N/\kappa,N]})\le \pi(\kappa)$ for any $M\in K$.

For any given $M\in K$, we can embed $\D_{g(M)\cap [N/\kappa,N]}$ into $\R^{\pi(\kappa)}$ by mapping an integer $m$ to its tuple of exponents of small primes. Applying the prime number theorem completes the proof.
\end{proof}

The main result of this paper is to provide a nontrivial bound on $\dim(\D_{g(M)\cap [N/\kappa,N]})$ using new results on multisets. We first review some notation for multisets.

When referring to multisets, we always count size/cardinality with multiplicity. We denote the set of distinct elements of a multiset $S$ by $\supp(S)$, so that $\dist{S}$ is the number of distinct elements of $S$. For multisets $S$ and $T$, we use $S\setminus T$ to denote the multiset where we subtract the multiplicity of an element in $T$ from its multiplicity in $S$ (without going below $0$). For example, $\supp(S\setminus T)$ is the set of elements which have greater multiplicity in $S$ than in $T$. Also, we write $S\subseteq T$ if the multiplicity of each element in $S$ is at most its multiplicity in $T$.

\section{Results on Multisets}\label{sec:multisets}

In this section we prove \cref{thm:multisets}, which is a modification of the main result of Brightwell et al. \cite{BKKT-paper}.

To prove this theorem, we will construct two sets of linear extensions of $\M_{[k,\ell]}^n$. The first set of extensions, $\LL_1$, will deal with all incomparable pairs of multisets $(S,T)$ for which $\dist{T\setminus S}\le 3(\ell-k).$ The second set of extensions, $\LL_2$, will deal with all incomparable pairs of multisets $(S,T)$ for which $\dist{T\setminus S}>3(\ell-k).$

\begin{lemma}\label{lem:L1}
Let $n$ be a positive integer and let $r$ be a positive real number with $r\ge 1$. Then there exists a set $\LL_1$ of at most $(3r+1)^2\log n$ linear extensions of $\M^n$ such that for every incomparable pair of multisets $(S,T)$ with $\dist{T\setminus S}\le 3r,$ there exists an extension $L\in\LL_1$ where $S>_LT$.
\end{lemma}

\begin{proof}

The proof uses the probabilistic method, with an argument similar to the proof of \cref{thm:weak-subsets} in \cite{Furedi-Khan-paper}.

Note that the result is clear if $n\le 3r+1$, so we will assume $n>3r+1.$

For a total order $\sigma$ on $[n]$, we define the lexicographic linear extension $L_\sigma$ of $\M^n$. Specifically, for $S,T\in \M^n$, we have $S>_{L_\sigma}T$ if $\max_\sigma(\supp(S\setminus T) \cup \supp(T\setminus S)) \in \supp(S\setminus T)$.



For a subset $Y$ of $[n]$ and an element $x\in[n]\setminus Y$, we say $x>_{\sigma}Y$ if $x>_{\sigma}y$ for all $y\in Y$. We would like to construct a set of at most $d=(3r+1)^2\log n$ choices of $\sigma$ such that for every $3r$-element subset $Y\subseteq[n]$ and $x\in[n]\setminus Y$, one of our $d$ choices of $\sigma$ gives $x>_{\sigma}Y$.

To accomplish this, consider sampling $d$ choices of $\sigma$ uniformly and independently at random from all possible total orders. For a given $3r$-element subset $Y\subseteq[n]$ and $x\in[n]\setminus Y$, the probability that $x>_{\sigma}Y$ in a given sample is $\frac{1}{3r+1}$. Thus, in $d$ samples, the probability that $x\not>_{\sigma}Y$ for any of the sampled $\sigma$ is $\left(1-\frac{1}{3r+1}\right)^d.$ The number of possible choices for $(x,Y)$ is at most $n^{3r+1}$. Thus the expected number of pairs $(x,Y)$ for which $x\not>_{\sigma}Y$ for all $d$ choices of $\sigma$ is at most $$n^{3r+1}\left(1-\frac{1}{3r+1}\right)^d<n^{3r+1}e^{-d/(3r+1)}=1 .$$
So, we can choose $d$ total orders $\sigma_1,\dots,\sigma_d$ such that for every choice of $(x,Y)$, $x>_{\sigma_i}Y$ for some chosen $\sigma_i$. Now, let $\LL_1=\{L_{\sigma_i}\colon 1\le i\le d\}$ be the set of lexicographic linear extensions corresponding to these choices of $\sigma$.



Consider an incomparable pair of multisets $(S,T)$ in $\M^n$ with $\dist{T\setminus S}\le 3r$. Choose any $x\in \supp(S\setminus T)$ and any $3r$-element subset $Y$ of $[n]\setminus\{x\}$ containing $\supp(T\setminus S)$. Then we have that $x>_{\sigma_i}Y$ for some chosen $\sigma_i$. In the corresponding linear extension $L_{\sigma_i}$, we will have $S>_{L_{\sigma_i}}T$, as desired.
\end{proof}

Note that \cref{lem:L1} holds for any value of $r$, but we will apply it with $r=\ell-k$. Additionally, note that since $\M_{[k,\ell]}^n$ is a suborder of $\M^n$, we can restrict the linear extensions we obtain from \cref{lem:L1} to linear extensions of $\M_{[k,\ell]}^n$ with the same property.

We construct $\LL_2$ from (ordered) partitions of $[n]$. We can represent a partition of $[n]$ with $a$ parts by a function $f\colon [n]\to [a]$. More generally, we represent a sequence of $t$ such partitions by a function $f\colon [n]\times[t]\to [a]$.

Working with these partitions instead of $[n]$ allows us to have a smaller object to deal with. However, we also need our sequence of partitions to remember some of the structure of $[n]$. This is encoded with the following property.

\begin{definition}[\cite{BKKT-paper}]
A function $f\colon [n]\times[t]\to [a]$ is $(a,b,r,t,n)$-good if for each subset $X\subseteq [n]$ of size $\abs{X}=b$, there exists $\tau\in[t]$ such that $\abs{f(X,\tau)}>r$.
\end{definition}

In other words, for each subset $X$ of $[n]$ with $b$ elements, we need to be able to choose one of our $t$ partitions that divides $X$ into more than $r$ parts. The following lemma guarantees the existence of good functions for appropriately chosen parameters.

\begin{lemma}[\cite{BKKT-paper}]\label{lem:good-func}
Let $a,b,r,t,n$ be positive integers such that $r<b\le n$ and $r<a$. If  $$\binom{n}{b}e^{rt}(r/a)^{(b-r)t}<1,$$ then there exists an $(a,b,r,t,n)$-good function.
\end{lemma}

The proof of this lemma is by the probabilistic method; we choose $f$ uniformly at random. See \cite[Lemma 2.2]{BKKT-paper} for a full proof. 

We are now ready to construct $\LL_2$. The key idea is to focus on a subset $R\subseteq[n]$ (which we will choose later) and use it to order our multisets by counting only elements of $R$ (with multiplicity). The goal is to find a collection $\mathcal{R}$ of several different subsets of $[n]$ such that for any incomparable pair of multisets $(S,T)$ with $\supp(T\setminus S)> 3(\ell-k)$, there exists an $R\in \mathcal{R}$ inducing an order with $S>T$. We accomplish this by taking several partitions of $[n]$, corresponding to a good function given by \cref{lem:good-func}. The fact that $T$ can have at most $\ell-k$ more elements than $S$ (since $k\le \abs{S},\abs{T}\le \ell$) will allow us to find a suitable collection $\mathcal{R}$ of relatively small size.

We will find $\mathcal{R}$ using a good function. Specifically, for an $(a,b,r,t,n)$-good function $f$ and $\alpha\in[a],\tau\in[t]$, we let $R_{\alpha,\tau}=\{i\in [n] \mid f(i,\tau)=\alpha\}$. That is, $R_{\alpha,\tau}$ denotes part $\alpha$ of partition $\tau$ in the sequence of partitions represented by $f$. We then let $\mathcal{R}=\{R_{\alpha,\tau} \mid \alpha\in[a],\tau\in[t]\}$.

\begin{lemma}\label{lem:L2}
Let $n,k,\ell,r$ be positive integers with $k<\ell$ and $r=\ell-k$. Then there exists a set $\LL_2$ of at most $18r\log n$ linear extensions of $\M_{[k,\ell]}^n$ such that for every incomparable pair of multisets $(S,T)$ with $\dist{T\setminus S}> 3r,$ there exists an extension $L\in\LL_2$ where $S>_LT$.
\end{lemma}

\begin{proof}

If $3r>n$, then $\dim (\M_{[k,\ell]}^n)\le n<18r\log n$. So assume $3r\le n$ and set $a=b=3r$, $t=3\log n$. Note that the condition in \cref{lem:good-func} holds, since
$$\binom{n}{b}e^{rt}(r/a)^{(b-r)t}=\binom{n}{3r}e^{rt}3^{-2rt}<n^{3r}e^{-rt}=1.$$
So, we fix an $(a,b,r,t,n)$-good function $f$. Now we will use $R_{\alpha,\tau}$ as described above.

For each $\alpha\in[a],\tau\in[t],j\in[2]$, we will construct a linear extension $L_{\alpha,\tau,j}$ of $\M_{k,l}^n$. That is, we will construct two extensions for each part in each partition. Our set of extensions $\LL_2$ will consist of these $2at$ extensions.

Fix two linear extensions $M_1$ and $M_2$ of $\M^n$ which both order multisets by size, but order the multisets of a given size in opposite orders. Also let $M_0$ be an arbitrary linear extension of $\M^n$.

For a multiset $S$, let $S_{\alpha,\tau}$ be the multiset obtained by restricting $S$ to the elements of $R_{\alpha,\tau}$, i.e., $S_{\alpha,\tau}=\{i\in S \mid f(i,\tau)=\alpha\}$. Note that $S_{\alpha,\tau}$ keeps the same multiplicities of elements as $S$.

We now construct our linear extensions $L_{\alpha,\tau,j}$. We let $S<_{L_{\alpha,\tau,j}}T$ if $S_{\alpha,\tau}<_{M_j}T_{\alpha,\tau}$. If $S_{\alpha,\tau}=T_{\alpha,\tau}$ then we let $S<_{L_{\alpha,\tau,j}}T$ if and only if $S<_{M_0}T$.

We now show that $\LL_2$ satisfies the desired condition. Suppose that $(S,T)$ is an incomparable pair of multisets in $\M_{[k,\ell]}^n$ with $\dist{T\setminus S}>3r=b$. Since $f$ is $(a,b,r,t,n)$-good, there exists $\tau\in[t]$ such that $\abs{f(T\setminus S,\tau)}>r$ holds\footnote{We let $f(X,\tau)=f(\supp(X),\tau)$ for multisets $X$.}.

If there exists $\alpha\in[a]$ such that $\abs{S_{\alpha,\tau}}>\abs{T_{\alpha,\tau}}$, then $S>T$ in both $L_{\alpha,\tau,1}$ and $L_{\alpha,\tau,2}$. So assume that $\abs{S_{\alpha,\tau}}\le\abs{T_{\alpha,\tau}}$ for all $\alpha\in[a]$. Next, note that 
$$\sum_{\alpha\in[a]}\abs{S_{\alpha,\tau}}=\abs{S},\sum_{\alpha\in[a]}\abs{T_{\alpha,\tau}}=\abs{T}.$$
Recall that $k\le \abs{S},\abs{T}\le \ell$, so $\abs{T}-\abs{S}\le \ell-k=r.$ Hence there are at most $r$ values of $\alpha\in[a]$ for which $\abs{S_{\alpha,\tau}}<\abs{T_{\alpha,\tau}}$. Thus there exists $\alpha\in f(T\setminus S,\tau)$ with $\abs{S_{\alpha,\tau}}=\abs{T_{\alpha,\tau}}$. Additionally, we can't have $S_{\alpha,\tau}=T_{\alpha,\tau}$ since $\alpha\in f(T\setminus S,\tau)$. So $S>T$ in either $L_{\alpha,\tau,1}$ or $L_{\alpha,\tau,2}$, and we are done.
\end{proof}
Applying \cref{lem:L1} with $r=\ell-k$ and \cref{lem:L2}, we have that $$\dim(\M_{[k,\ell]}^n)\le (3r+1)^2\log n+18r\log n \le 34(\ell-k)^2\log n.$$
This proves \cref{thm:multisets}.




\section{Weighted Multisets}\label{sec:weighted-multisets}







In this section we prove \cref{thm:weighted-multisets} using the same proof strategy as for \cref{thm:multisets}. This time we are given a value of $r$ which depends somewhat on $\ell-k$. We will again construct two sets of extensions of $\M_{[k,\ell]}^{n,\vv}$. The first set of extensions, $\LL_1$, will deal with all incomparable pairs of multisets $(S,T)$ for which $\dist{T\setminus S}\le 3r$, and the second set of extensions, $\LL_2$, will deal with all incomparable pairs of multisets $(S,T)$ for which $\dist{T\setminus S}>3r.$

To construct $\LL_1$, we apply \cref{lem:L1} and restrict the resulting linear extensions to $\M_{[k,\ell]}^{n,\vv}$. This gives us at most $(3r+1)^2\log n$ linear extensions of $\M_{[k,\ell]}^{n,\vv}$ such that for every incomparable pair of multisets $(S,T)$ with $\dist{T\setminus S}\le 3r,$ there exists an extension $L\in\LL_1$ where $S>_LT$.

To construct $\LL_2$, we will use the same strategy as in the proof of \cref{lem:L2}. However, we will need some new ideas. Previously, for an incomparable pair of multisets $(S,T)$, if a subset $R_{\alpha,\tau}$ of $[n]$ ordered $T$ above $S$, then this would ``cost'' $1$ from $\abs{T}-\abs{S}\le \ell-k$. However, when dealing with $\vv$-size, we can no longer make such a claim, since $\abs{T_{\alpha,\tau}}_{\vv}-\abs{S_{\alpha,\tau}}_{\vv}$ could be very small. Instead, we will find a way to still ``win'' if $\abs{T_{\alpha,\tau}}_{\vv}-\abs{S_{\alpha,\tau}}_{\vv}$ is sufficiently small.

\begin{lemma}\label{lem:slant-L2}
Let $\vv\in\R^n_{>0}$, and $k,\ell,r$ be real numbers with $0<k<\ell$ and $r\ge1$. Suppose that $m(\vv,r)\ge 2(\ell-k)$. Then there exists a set $\LL_2$ of at most $27r\log n$ linear extensions of $\M_{[k,\ell]}^n$ such that for every incomparable pair of multisets $(S,T)$ with $\dist{T\setminus S}> 3r,$ there exists an extension $L\in\LL_2$ where $S>_LT$.
\end{lemma}

\begin{proof}

Set $a=b=3r$, $t=3\log n$, and fix an $(a,b,r,t,n)$-good function $f$ as in the proof of \cref{lem:L2}. Let $R_{\alpha,\tau}=\{i\in [n] \mid f(i,\tau)=\alpha\}$ denote part $\alpha$ of partition $\tau$.

For each $\alpha\in[a]$, $\tau\in[t]$, and $j\in\{0,1,2\}$, we will construct a linear extension $L_{\alpha,\tau,j}$ of $\M_{[k,\ell]}^{n,\vv}$. That is, we will construct three extensions for each part in each partition. Our set of extensions $\LL_2$ will consist of these $3at$ extensions.

For a multiset $S$, define $S_{\alpha,\tau}$ as before. Let $M_1$ and $M_2$ be two linear extensions of $\M^n$ which both order multisets by $\vv$-size, but order the multisets of a given $\vv$-size in opposite orders. Let $M_0$ be an arbitrary linear extension of $\M^n$.

First we define $L_{\alpha,\tau,0}$. For each $\alpha,\tau$, we will have $L_{\alpha,\tau,0}$ order multisets $S$ by $\abs{S_{\alpha,\tau}}_{\vv}$, breaking ties using $M_0$. That is, we have $S<_{L_{\alpha,\tau,0}}T$ if $\abs{ S_{\alpha,\tau} }_{\vv} <\abs{ T_{\alpha,\tau} }_{\vv}$. If $\abs{ S_{\alpha,\tau} }_{\vv} =\abs{ T_{\alpha,\tau} }_{\vv}$, then $S<_{L_{\alpha,\tau,0}}T$ if and only if $S<_{M_0}T$.

We now define $L_{\alpha,\tau,j}$ for $j\in\{1,2\}$. For each $\alpha,\tau$, both extensions will order multisets $S$ by taking into account the $\vv$-size $\abs{S_{\alpha,\tau}}_{\vv}$. Specifically, $L_{\alpha,\tau,j}$ will order multisets $S$ by applying an ordering $K_{\alpha,\tau,j}$ of $\R$ to $\abs{S_{\alpha,\tau}}_{\vv}$ as follows. If $\abs{S_{\alpha,\tau}}_{\vv}<_{K_{\alpha,\tau,j}}\abs{T_{\alpha,\tau}}_{\vv}$, then $S<_{L_{\alpha,\tau,j}}T$. This defines how $L_{\alpha,\tau,j}$ orders multisets $S$ and $T$ except when $\abs{S_{\alpha,\tau}}_{\vv}=\abs{T_{\alpha,\tau}}_{\vv}$. We will deal with this edge case later.

To define $K_{\alpha,\tau,j}$, we first let $\varepsilon=\varepsilon_{\vv,f}(\alpha,\tau)=\min_{i\in R_{\alpha,\tau}}(v_i)$, where $v_i$ is the $i$th entry of $\vv$. Now, consider dividing the real number line into half-open intervals of length $\varepsilon$ in the following two ways.
$$\R=\cdots\cup[0,\varepsilon)\cup[\varepsilon,2\varepsilon)\cup[2\varepsilon,3\varepsilon)\cup\cdots$$
$$\R=\cdots\cup[\varepsilon/2,3\varepsilon/2)\cup[3\varepsilon/2,5\varepsilon/2)\cup[5\varepsilon/2,7\varepsilon/2)\cup\cdots$$
We let $K_{\alpha,\tau,1}$ order the intervals in the first partition of $\R$ in increasing order, but order the elements of each interval in decreasing order. We define $K_{\alpha,\tau,2}$ in the same way from the second partition of $\R$.
Specifically, for reals $x<y$, we have $x<_{K_{\alpha,\tau,1}}y$ if and only if $\floor{x/\varepsilon}<\floor{y/\varepsilon}$, and $x<_{K_{\alpha,\tau,2}}y$ if and only if $\floor{x/\varepsilon-1/2}<\floor{y/\varepsilon-1/2}$.

This in turn defines our linear extensions $L_{\alpha,\tau,1}$ and $L_{\alpha,\tau,2}$. For example, for multisets $S$ and $T$ with $\varepsilon<\abs{S_{\alpha,\tau}}_{\vv}<3\varepsilon/2<\abs{T_{\alpha,\tau}}_{\vv}<2\varepsilon$, we will have $S>_{L_{\alpha,\tau,1}}T$ and $S<_{L_{\alpha,\tau,2}}T$. Additionally, note that if $S$ and $T$ are two multisets satisfying $\abs{T_{\alpha,\tau}}_{\vv}-\varepsilon/2<\abs{S_{\alpha,\tau}}_{\vv}<\abs{T_{\alpha,\tau}}_{\vv}$, then at least one of $L_{\alpha,\tau,1}$ and $L_{\alpha,\tau,2}$ orders $S$ greater than $T$.

Now we deal with the edge case of distinct multisets $S,T$ with $\abs{S_{\alpha,\tau}}_{\vv}=\abs{T_{\alpha,\tau}}_{\vv}$. If $S_{\alpha,\tau}\ne T_{\alpha,\tau}$, then $S<_{L_{\alpha,\tau,j}}T$ if and only if $S_{\alpha,\tau}<_{M_j}T_{\alpha,\tau}$. Finally, if $S_{\alpha,\tau}= T_{\alpha,\tau}$, then $S<_{L_{\alpha,\tau,j}}T$ if and only if $S<_{M_0}T$.

Next, we check that $L_{\alpha,\tau,j}$ is indeed a linear extension of $\M_{[k,\ell]}^{n,\vv}$. Suppose $S,T\in\M_{[k,\ell]}^{n,\vv}$ with $S\subseteq T$. Then $S\le_{M_0}T$. Also, $S_{\alpha,\tau}\subseteq T_{\alpha,\tau}$, so $S_{\alpha,\tau}\le_{M_j}T_{\alpha,\tau}$. Thus $S\le_{L_{\alpha,\tau,j}} T$ if $\abs{S_{\alpha,\tau}}_{\vv}=\abs{T_{\alpha,\tau}}_{\vv}$. Now suppose that $\abs{S_{\alpha,\tau}}_{\vv}<\abs{T_{\alpha,\tau}}_{\vv}$. Pick any $i\in (T\setminus S)_{\alpha,\tau}$. Since $S\subseteq T$, we have $$\abs{T_{\alpha,\tau}}_{\vv}-\abs{S_{\alpha,\tau}}_{\vv}=\abs{(T\setminus S)_{\alpha,\tau}}_{\vv}\ge v_i\ge \varepsilon(\alpha,\tau).$$
Thus $\abs{S_{\alpha,\tau}}_{\vv}$ and $\abs{T_{\alpha,\tau}}_{\vv}$ cannot lie in the same half-open interval of length $\varepsilon(\alpha,\tau)$. So $\abs{S_{\alpha,\tau}}_{\vv}<_{K_{\alpha,\tau,j}}\abs{T_{\alpha,\tau}}_{\vv}$ and $S<_{L_{\alpha,\tau,j}} T$. Therefore $L_{\alpha,\tau,j}$ is indeed a linear extension of $\M_{[k,\ell]}^{n,\vv}$.




We now show that $\LL_2$ satisfies the desired condition. Suppose that $(S,T)$ is an incomparable pair of multisets in $\M_{[k,\ell]}^{n,\vv}$ with $\dist{T\setminus S}>3r=b$. Then $k\le \abs{S}_{\vv},\abs{T}_{\vv}\le \ell$. So $\abs{T}_{\vv}-\abs{S}_{\vv}\le \ell-k.$ Since $f$ is $(a,b,r,t,n)$-good, there exists $\tau\in[t]$ such that $\abs{f(T\setminus S,\tau)}>r$. If there exists $\alpha\in[a]$ such that $\abs{S_{\alpha,\tau}}_{\vv}>\abs{T_{\alpha,\tau}}_{\vv}$, then we are done. So assume that $\abs{S_{\alpha,\tau}}_{\vv}\le\abs{T_{\alpha,\tau}}_{\vv}$ for all $\alpha\in[a]$. Next, consider $\alpha\in f(T\setminus S,\tau)$. If $\abs{S_{\alpha,\tau}}_{\vv}>\abs{T_{\alpha,\tau}}_{\vv}-\varepsilon(\alpha,\tau)/2$, then $S>T$ in $L_{\alpha,\tau,j}$ for at least one $j\in\{1,2\}$. So assume that $\abs{T_{\alpha,\tau}}_{\vv}\ge \abs{S_{\alpha,\tau}}_{\vv}+\varepsilon(\alpha,\tau)/2$ for each $\alpha\in f(T\setminus S,\tau)$. Summing over all such $\alpha$, we obtain

$$\ell-k\ge \abs{T}_{\vv}-\abs{S}_{\vv}\ge \frac{1}{2}\sum_{\alpha\in f(T\setminus S,\tau)}\varepsilon(\alpha,\tau)> \frac{1}{2}m(\vv,r),$$
where the last inequality follows by noting that each $\varepsilon(\alpha,\tau)$ is a distinct entry of $\vv$ and using $\abs{f(T\setminus S,\tau)}>r$. Since $m(\vv,r)\ge 2(\ell-k)$, we have a contradiction. This completes the proof.

\end{proof}

Applying \cref{lem:L1} and \cref{lem:slant-L2}, we have that $$\dim(\M_{[k,\ell]}^{n,\vv})\le (3r+1)^2\log n+27r\log n \le 43r^2\log n.$$
This proves \cref{thm:weighted-multisets}.



\section{Results on Divisibility in $\N$}\label{sec:divisibility}

In this section we use \cref{thm:weighted-multisets} to obtain a significantly better bound on $\dim(\D_{[N/\kappa,N]})$, proving \cref{thm:divisibility}.

We use some of the ideas from the proof of \cref{prop:easy-divisibility}. As before, call a prime $p$ \emph{small} if $p\le \kappa$ and \emph{large} if $p>\kappa$. Let $K$ be the set of all positive integer $M\le N$ with no small primes dividing $M$. For each $M\in K$, let $g(M)$ be the set of all positive integers of the form $Mq$, where $q$ has only small prime divisors. Now, recall that to bound $\dim(\D_{[N/\kappa,N]})$ it suffices to bound $\dim(\D_{g(M)\cap [N/\kappa,N]})$ for each $M\in K$. We will accomplish this by applying \cref{thm:weighted-multisets}.

Let $n=\pi(\kappa)$, and let $\vv=(\log2,\log3,\dots,\log p_n)$, where $p_i$ is the $i$th prime. Fix some $M\in K$ and let $k_M=\log((N/\kappa)/M)$ and $\ell_M=\log(N/M)$, so that $\ell_M-k_M=\log(\kappa)$. Now, by associating an integer with its prime factorization we have an isomorphism between posets $$\M_{[k_M,\ell_M]}^{n,\vv}\cong\D_{g(M)\cap [N/\kappa,N]}.$$


If $\kappa<3$ then $\dim(\M_{[k_M,\ell_M]}^{n,\vv})\le 2$. Otherwise, it can be checked (see \cref{app:prime-bash}) that with $r=4\frac{\log \kappa}{\log \log \kappa}$, we have $m(\vv,r)\ge 2\log(\kappa)$. So by \cref{thm:weighted-multisets}, $\dim(\M_{[k_M,\ell_M]}^{n,\vv})\le 43r^2\log n$. This holds for all choices of $M$, so we have that $$\dim(\D_{[N/\kappa,N]})\le 43r^2\log n\le 688\frac{(\log\kappa)^3}{(\log\log \kappa)^2},$$
which proves \cref{thm:divisibility}.

\section{Other Divisibility Orders}\label{sec:polynomial}

In this section we consider other divisibility orders. In some sense the poset $\M_{[k,\ell]}^{n,\vv}$ captures all possible normed divisibility orders subject to reasonable conditions. Specifically, given a (multiplicative) monoid with a norm and some notion of primes and unique factorization\footnote{We just need commutativity, cancellation, and finitely many elements with norm at most $N$ for each $N$.}, we can construct a divisibility order on the elements with norms lying in an interval $[N/\kappa,N]$. Letting $\vv$ be the vector of the norms of primes with norm at most $\kappa$, we can then apply \cref{thm:weighted-multisets}. If we have an understanding of the primes and the norm we can obtain a concrete bound, as in \cref{thm:divisibility}.

As another example of this principle, we analyze the divisibility order on monic polynomials over a finite field. As in the proof of \cref{prop:easy-divisibility}, $P(\F_q)_{[d_0-\delta,d_0]}$ is isomorphic to a disjoint union of posets of the form $\M^{n,\vv}_{k,\ell},$ where $\ell-k=\delta$, $n$ is the number of monic irreducible polynomials over $\F_q[x]$ with degree at most $\delta$, and $\vv\in \R^n$ is the vector of degrees of these $n$ polynomials.

It can be checked (see \cref{app:polynomial-bash}) that $n\le q^\delta$ and that with $r=4.6\frac{\delta\log q}{\log \delta}$, we have $m(\vv,r)\ge 2\delta$. Additionally, since each entry of $\vv$ is at least $1$, we have $m(v,r)\ge 2\delta$ with $r=2\delta$. These two choices of $r$ handle the regimes when $q<\delta$ and $q\ge \delta$, respectively. By applying \cref{thm:weighted-multisets}, we obtain
$$\dim(P(\F_q)_{[d_0-\delta,d_0]})\le 43\left(\min\left(4.6\frac{\delta\log q}{\log \delta},2\delta\right)\right)^2\log(q^\delta)\le \min\left(910\frac{(\delta\log q)^3}{(\log \delta)^2},172\delta^3\log q\right).$$
This proves \cref{thm:polynomial}.

\section{Further Directions}\label{sec:further-directions}

Considering multiset posets was initially motivated by working on divisibility posets. However, multiset posets are a natural extension of subset posets and interesting on their own. It would be nice to see \cref{thm:weighted-multisets} applied to problems not directly obtained from divisibility orders.

Although \cref{thm:divisibility} is a substantial improvement over previously known bounds, it is not quite tight with the lower bound given by Lewis and Souza. This lower bound cannot be improved without new ideas, since it is obtained by embedding the optimal subset poset of the form $\Q_{[1,\ell]}^n$ into $\D_{[N/\kappa,N]}$ \cite{Lewis-Souza-paper}. On the other hand, there are also reasons to believe that \cref{thm:divisibility} might be close to tight. The limiting factor in the upper bound is the dimension of the subset poset $\Q_{[1,3r]}^n$. In the regime when $r$ is close to $\log n$, Kierstead showed that $\dim(\Q_{[1,3r]}^n)$ is close to $r^3$ \cite{Kierstead-paper}. Unfortunately the proof technique does not extend easily to the divisibility problem, but perhaps it could be modified with new ideas.

\section{Acknowledgements}\label{sec:acknowledgements}


I would like to thank Noah Kravitz for helpful conversations and detailed feedback on earlier versions of this paper. Additionally, I am grateful to Victor Souza for bringing other divisibility orders to my attention, among other helpful comments. I would also like to thank Michael Ren and William Trotter for helpful conversations and Joe Gallian, Carl Schildkraut, Zach Hunter, and the anonymous reviewers for comments on earlier versions of this paper. Finally, I want to express my gratitude to Joe Gallian, Amanda Burcroff, Colin Defant, Noah Kravitz, and Yelena Mandelshtam for organizing the Duluth REU and the project suggestion.

\section{Statements}

\textbf{Funding Statement.} This research was conducted at the University of Minnesota Duluth Mathematics REU and was supported, in part, by NSF-DMS Grant 1949884 and NSA Grant H98230-20-1-0009. Additional support was provided
by the CYAN Mathematics Undergraduate Activities Fund.

\textbf{Data Availability Statement.} This paper does not make use of any data.

\textbf{Conflict of Interest Statement.} The author is not aware of any relevant conflict of interest.

\textbf{Author Contribution Statement.} This is a single-author paper. All contributions are due to Milan Haiman. \cref{sec:acknowledgements} includes acknowledgements.

\bibliographystyle{plain}
\bibliography{refs}

\appendix

\section{Computations for \cref{thm:divisibility}}\label{app:prime-bash}

In this appendix we show the following bound. 

\begin{proposition}
Let $\vv=(\log 2,\log 3, \dots, \log p_n)$, where $p_i$ is the $i$th prime. Let $\kappa\ge 3$ be a real number, and let $r=4\frac{\log \kappa}{\log \log \kappa}$. Then $m(\vv,r)\ge 2\log(\kappa)$.
\end{proposition}

\begin{proof}

We use \cite{Robin-bounds}, which provides bounds on the Chebyshev function $\vartheta(x)=\sum_{p\le x}\log p$ (where the sum is over primes $p$). Note that $m(\vv,r)=\vartheta(p_{\floor{r}})$. By \cite[Theorem 6]{Robin-bounds}, we have that $$\vartheta(p_{\floor{r}})\ge \floor{r}(\log \floor{r}+\log\log\floor{r}-1.076869).$$
Now, let $\lambda=\log \kappa>1$ and note that $\frac{\lambda}{\log\lambda}\ge e$. So $$\floor{r}>r-1> \left(4-\frac{1}{e}\right)\frac{\lambda}{\log \lambda}\ge3.5\frac{\lambda}{\log \lambda}.$$
Let $\eta=1.076869$ and note that $\eta<\log (3.5)$. Using the above bounds and $\frac{\log\log\lambda}{\log\lambda}\le\frac{1}{e}$, we have that 

\begin{align*}
m(\vv,r)&>3.5\frac{\lambda}{\log \lambda}\left(\log\left(3.5\frac{\lambda}{\log \lambda}\right)+\log\log\left(3.5\frac{\lambda}{\log \lambda}\right)-\eta\right)\\
&>3.5\frac{\lambda}{\log \lambda}\left(\log(3.5)+\log\lambda-\log\log\lambda-\eta\right)\\
&>3.5\frac{\lambda}{\log \lambda}\left(\log\lambda-\log\log\lambda\right)\\
&=2\lambda+\lambda\left(1.5-\frac{3.5\log\log\lambda}{\log\lambda}\right)\\
&\ge2\lambda+\lambda\left(1.5-\frac{3.5}{e}\right)\\
&>2\lambda.
\end{align*}
So $m(\vv,r)\ge 2\log\kappa$, as desired.

\end{proof}

\section{Computations for \cref{thm:polynomial}}\label{app:polynomial-bash}

In this appendix we show bounds on irreducible polynomials in $\F_q[x]$. Let $n_i$ be the number of monic irreducible polynomials of degree $i$. Recall that the product of all monic irreducible polynomials of degree dividing $i$ is $x^{q^i}-x$. In particular, this means that $n_i\le q^i/i$. In the notation of \cref{sec:polynomial}, we have that $n=n_1+\dots+n_{\delta}$ and $\vv\in \R^n$ is a vector with $n_i$ entries being $i$.

\begin{proposition}
We have $n\le q^\delta$.
\end{proposition}

\begin{proof}
We have $n_1=q$ and for $i\ge2$,$$n_i\le q^i\cdot\frac{1}{i}\le q^i\cdot\frac{q-1}{q}=q^i-q^{i-1}.$$
So $$n=n_1+\dots+n_{\delta}\le q+(q^2-q)+\dots+(q^\delta-q^{\delta-1})=q^\delta.$$
\end{proof}

\begin{proposition}
Let $r=4.6\frac{\delta\log q}{\log \delta}$. Then $m(\vv,r)\ge 2\delta$.
\end{proposition}

\begin{proof}
For each $i$, $\vv$ has $n_1+\dots+n_{i-1}$ entries that are less than $i$. Thus $m(\vv,r)$ has $$\floor{r}-(n_1+\dots+n_{i-1})$$ terms that are at least $i$. Then for any $d\in \N$ we have $$m(\vv,r)\ge \sum_{i=1}^d\floor{r}-(n_1+\dots+n_{i-1})=d\floor{r}-(d-1)n_1-\dots-(1)n_{d-1}.$$
We will apply this with $d=\floor{\log \delta/\log q}+1$. Since $\frac{\delta}{\log \delta}\ge e$ and $\log q\ge \log 2$, $$\floor{r}>r-1\ge 4\frac{\delta\log q}{\log \delta}+0.6\cdot e\log 2-1>4\frac{\delta\log q}{\log \delta}.$$
Now we have $$m(\vv,r)> \frac{\log \delta}{\log q}\cdot4\frac{\delta\log q}{\log \delta}-(d-1)\frac{q^1}{1}-\dots-(1)\frac{q^{d-1}}{d-1}.$$
By rearrangement, $$(d-1)\frac{q^1}{1}+\dots+(1)\frac{q^{d-1}}{d-1}\le (1)\frac{q^1}{1}+\dots+(d-1)\frac{q^{d-1}}{d-1}=q^1+\dots+q^{d-1}\le 2q^{d-1}.$$
So we have $$m(\vv,r)>4\delta-2q^{d-1}\ge 4\delta-2\delta=2\delta.$$

\end{proof}

\end{document}